\documentclass[12pt,thmsa,sw20lart]{article}
\usepackage{amsmath,amsthm,amssymb,epsfig,graphics,subfigure, wrapfig, epsfig, color}
\usepackage{tikz}
\input amssym.def
\input amssym.tex
\date{}
\newtheorem{lemma}{Lemma}
\newtheorem{proposition}{Proposition}
\newtheorem{theorem}{THEOREM}

\theoremstyle{definition}
\newtheorem{example}{Example}
\newtheorem{remark}{Remark}

\allowdisplaybreaks

\begin{document}

\title{{\bf A new result on boundedness of the Riesz potential in central Morrey--Orlicz spaces}}
\author {Evgeniya Burtseva and Lech Maligranda}

\maketitle


\begin{abstract}
\noindent {\footnotesize We improve our results on boundedness of the Riesz potential in the central Morrey--Orlicz 
spaces and the corresponding weak-type version. We also present two new properties of the central Morrey--Orlicz spaces: 
nontriviality and inclusion property.}
\end{abstract}

%%%%%%%%%%%%%%%%%%%%%%%%%%%%%%%%%%%%%%%%%%%%%%%Section 1
\section{Central Morrey--Orlicz spaces} \label{sec:central}
%%%%%%%%%%%%%%%%%%%%%%%%%%%%%%%%%%%%%%%%%%%%%%%

A function $\Phi \colon [0,\infty) \to [0,\infty]$ is called a {\it Young function}, if it is a nondecreasing convex 
function with $\lim_{u \to 0^+}\Phi(u) = \Phi(0) = 0$, and not identically $0$ or $\infty$ in $(0,\infty)$. 
It may have jump up to $\infty$ at some point $u>0$, but then it should be left continuous at $u$. 

To each Young's function $\Phi$ one can associate another convex function $\Phi^*$, i.e., the {\it complementary function} 
to $\Phi$, which is defined by 
$$
\Phi^* (v) = \sup_{u>0} \, [uv - \Phi(u)] ~~ {\rm for} ~~ v \geq 0. 
$$
Then $\Phi^*$ is also a Young's function and $\Phi^{**} = \Phi$. Note that $u \leq \Phi^{-1}(u) \Phi^{*^{-1}}(u) \leq 2 u$ 
for all $u > 0,$ where $\Phi^{-1}$ is the right-continuous inverse of $\Phi$ defined by 
$$
\Phi^{-1}(v) = \inf \{u \ge 0 \colon \Phi(u)>v\} ~ {\rm with} ~ \inf \emptyset = \infty.
$$ 

We say that Young's function $\Phi $ satisfies the $\Delta_2$-{\it condition} and we write shortly $\Phi \in \Delta_2$, if 
$0< \Phi(u) < \infty$ for $u>0$ and there exists a constant $D_2 > 1$ such that 
\begin{equation*} \label{D2}
\Phi(2u) \leq D_2 \Phi(u) ~~ {\rm for ~all} ~~ \ u > 0. 
\end{equation*}

For any Young's function $\Phi$, the number $\lambda \in \mathbb R$ and an open ball 
$B_r = \{x \in \mathbb{R}^n \colon |x|<r\}, r>0$ we can define {\it central Morrey--Orlicz spaces $M^{\Phi,\lambda}(0)$} 
as all $f \in L^1_{loc}({\mathbb R^n})$ such that 
$$
\|f\|_{M^{\Phi, \lambda}(0)} = \sup_{r > 0} \|f\|_{\Phi, \lambda, B_r} < \infty, 
$$
where
$$
\|f\|_{\Phi, \lambda, B_r} = \inf \big\{ \varepsilon > 0\colon \dfrac 1 {|B_r|^{\lambda}} \int_{B_r} \Phi\big(\dfrac {|f(x)|} \varepsilon \big) 
\,dx \leq 1 \big\}. 
$$
Similarly, the {\it weak central Morrey--Orlicz spaces $WM^{\Phi,\lambda}(0)$} are defined as
\begin{equation*} \label{wCMOnorm}
WM^{\Phi,\lambda}(0) = \left\{ f \in L^1_{loc}({\mathbb R^n})\colon \|f\|_{WM^{\Phi,\lambda}(0)} = 
\sup_{r > 0} \|f\|_{\Phi, \lambda, B_r, \infty} < \infty \right\}, 
\end{equation*}
where
$$
\|f\|_{\Phi, \lambda, B_r, \infty} = \inf \big\{ \varepsilon > 0 \colon \sup_{u>0} \Phi(\dfrac{u}{\varepsilon}) \, \dfrac 1 {|B_r|^{\lambda}} 
\, d(f\chi_{B_r}, u) 
\leq 1\big\}, 
$$
and $d(f, u) = | \{x \in {\mathbb R^n} \colon |f(x)| > u \}|$. 

\vspace{2mm}

The properties of these spaces can be found in \cite{BMM2022}. If $\Phi(u) = u^p,\,1\le p<\infty$ and $\lambda \in \mathbb{R},$ 
then $M^{\Phi,\lambda}(0) = M^{p,\lambda}(0)$ and $WM^{\Phi,\lambda}(0) = WM^{p,\lambda}(0)$ are classical central and 
weak central Morrey spaces. Moreover, for $\lambda = 0$ the spaces $M^{\Phi, 0}(0) = L^{\Phi}(\mathbb R^n)$ and 
$WM^{\Phi, 0}(0) = WL^{\Phi}(\mathbb R^n)$ are classical Orlicz and weak Orlicz spaces.

In the following lemma and later, $B(x_0, r_0)$ will denote an open ball with the center at $x_0 \in \mathbb R^n$ and 
radius $r_0 > 0$, that is, $B(x_0, r_0) = \{x \in {\mathbb R^n}\colon |x - x_0| < r_0\}$. 

%%%%%%%%%%%%%%%%%%%%%%%%%%%%%%%. Lemma 1
\begin{lemma}\label{lemma1} Let $\Phi$ be a Young's function, $\Phi^*$ its complementary function, $0 \le \lambda \le 1$ 
and $r > 0$. Then
\begin{itemize} 
\item[$({\rm i})$] $\int_{B_r} |f(x)g(x)|\,dx \leq 2\, |B_r|^{\lambda}\, \|f\|_{\Phi, \lambda, B_r} \|g\|_{\Phi^*, \lambda, B_r}$.
\item[$({\rm ii})$] $\| \chi_{B(x_0, r_0)}\|_{\Phi^{\ast}, \lambda, B_r} \le \frac{|B_r \cap B(x_0, r_0)|}{|B_r|^\lambda}\Phi^{-1} 
\left(\frac{|B_r|^\lambda}{|B_r \cap B(x_0, r_0)|} \right) $, {\it where $B_r \cap B(x_0, r_0) \neq \emptyset$ for $x_0 \in \mathbb 
R^n$ and $r_0 > 0$.}

In particular, $\|\chi_{B_r} \|_{\Phi^*, \lambda, B_r} \leq \dfrac {\Phi^{-1} \left( |B_r|^{\lambda-1} \right)} {|B_r|^{\lambda-1}}.$
\item[$({\rm iii})$] $\| \chi_{B_t}\|_{\Phi, \lambda, B_r} = 1/ {\Phi^{-1} \left(\frac{|B_r|^\lambda}{|B_r \cap B_t|}\right) }$ and 
$\|\chi_{B_t}\|_{M^{\Phi, \lambda} (0)} = \dfrac{1}{\Phi^{-1}(|B_t|^{\lambda - 1})}$ {\it for any $t > 0$.}
\end{itemize}
\end{lemma}
%%%%%%%%%%%%%%%%%%%%%

Proof of this lemma can be found in \cite[Lemma 1]{BMM2022}. 

%%%%%%%%%%%%%%%%%%%%%%%%%%%%%%%%%%%%%%%%%% Section 2
\section{Riesz potential in the central Morrey--Orlicz spaces} \label{sec:Imbed}
%%%%%%%%%%%%%%%%%%%%%%%%%%%%%%%%%%%%%%%%%%%%%%%

We will work with the central Morrey--Orlicz spaces, defined by the Orlicz functions. A function $\Phi \colon [0,\infty) \to [0,\infty)$ 
is called an {\it Orlicz function}, if it is a strictly increasing continuous and convex function with $\Phi(0)=0$.

Let $f\colon {\mathbb R^n} \rightarrow \mathbb R$ be a Lebesgue measurable function and $\alpha \in (0,n)$. 
The Riesz potential is defined as
\begin{equation*}\label{Riesz_potential}
I_\alpha f(x) = \int\limits_{\mathbb{R}^n} \frac{f(y)}{|x-y|^{n-\alpha}} \,dy,\,\, \text{for}\,\,x \in \mathbb{R}^n.
\end{equation*}

The linear operator $I_{\alpha}$ plays an important role in various branches of analysis, including potential theory, 
harmonic analysis, Sobolev spaces, partial differential equations and can be treated as a special singular integral. 
That is why it is important to study its boundedness between different spaces. Many authors investigated 
boundedness of $I_\alpha$ in Morrey, Orlicz and Morrey--Orlicz spaces. We present here our main theorem on 
the boundedness of the Riesz potential in the central Morrey--Orlicz spaces.

\vspace{2mm}

In order to prove our result we will use estimate from \cite{MM15} for the Hardy--Littlewood maximal operator in central 
Morrey--Orlicz spaces. The Hardy--Littlewood maximal operator $M$ or {\it centred maximal function} $Mf$ of a function 
$f$ defined on $\mathbb R^n$ is defined at each $x {\in \mathbb R^n}$ as
$$
Mf(x) = \sup\limits_{r>0} \frac {1} {|B(x,r)|} \int\limits_{B(x,r)} |f(y)| \,dy.
$$
For any Orlicz function $\Phi$ and $0 \leq \lambda \leq 1$, maximal operator $M$ is bounded on 
$M^{\Phi, \lambda}(0)$, provided ${\Phi}^* \in {\Delta}_2$, and then there exists a constant $C_0>1$ such that 
\begin{equation*}  
\|Mf\|_{M^{\Phi, \lambda}(0)} \le C_0 \,  \|f\|_{M^{\Phi, \lambda}(0)},  \quad \text{for all} ~ f \in M^{\Phi, \lambda}(0)
\end{equation*} 
(see \cite[Theorem 6(i)]{MM15}). Moreover, the maximal operator $M$ is bounded from  $M^{\Phi, \lambda}(0)$ to 
$WM^{\Phi, \lambda}(0)$, that is, there exists a constant $c_0>1$ such that $\|Mf\|_{WM^{\Phi, \lambda}(0)} \le c_0 \,  
\|f\|_{M^{\Phi, \lambda}(0)}$ for all $f \in M^{\Phi, \lambda}(0)$ (see \cite[Theorem 6(ii)]{MM15}).

\vspace{2mm}

Furthermore, in the proof of the main result we will use Hedberg's pointwise estimate from \cite[p. 506]{hed}.

%%%%%%%%%%%%%%%%%%%%%%%%%%%%%%%. Lemma 2
\begin{lemma}[Hedberg]\label{lemma2} If $f\colon {\mathbb R^n} \rightarrow \mathbb R$ is a Lebesgue measurable function 
and $\alpha \in (0,n)$, then for all $x \in \mathbb R^n$ and $r > 0$
\begin{equation*}\label{lemma2}
\int_{|y - x| \leq r} |f(y)| |x - y|^{\alpha - n} dy \leq C_H\, r^{\alpha} Mf(x),
\end{equation*}
with $C_H = \frac{2^n}{2^{\alpha} - 1} v_n$, where $v_n = |B(0, 1)| = \pi^{n/2}/\Gamma(n/2+1)$.
\end{lemma}

\begin{proof} For the sake of completeness, we include its proof, taking care about the constant $C_H$ in the estimate.
For any $x \in \mathbb R^n$ and $r > 0$
\begin{eqnarray*}
\int_{|y - x| \leq r} \frac{|f(y)|}{|x - y|^{n - \alpha}} dy 
&=&
\sum\limits_{m=0}^{\infty} \int_{r 2^{-m-1} < | y - x| \leq r 2^{-m}}  \frac{|f(y)|}{|x - y|^{n - \alpha}} dy \\
&\leq &
\sum\limits_{m=0}^{\infty} \int_{B(x, r 2^{-m}) \setminus B(x, r 2^{-m-1})}  \frac{|f(y)|}{(r 2^{-m-1})^{n - \alpha}} dy\\
&\leq&
2^{n - \alpha} r^{\alpha} \sum\limits_{m=0}^{\infty} 2^{-m \alpha} (r 2^{-m})^{-n}  \int_{B(x, r 2^{-m})} |f(y)| dy\\
&=&
2^{n - \alpha} r^{\alpha} \sum\limits_{m=0}^{\infty} 2^{-m \alpha} \frac{v_n}{| B(x, r 2^{-m})|}  \int_{B(x, r 2^{-m})} |f(y)| dy\\
&\leq&
2^{n - \alpha} r^{\alpha} \,v_n \sum\limits_{m=0}^{\infty} 2^{-m \alpha} Mf(x) = \frac{2^n \,v_n}{2^{\alpha} - 1} \, r^{\alpha} Mf(x).
\end{eqnarray*}
\end{proof}

%%%%%%%%%%%%% Thm 1
\begin{theorem}\label{th_bound}
Let $0<\alpha<n,$ $\Phi, \Psi$ be Orlicz functions and either $0 < \lambda, \mu <1, \lambda \neq \mu$ or $\lambda=0$ 
and $0 \leq \mu < 1.$ Assume that there exist constants $C_1, C_2 \ge 1$ such that 
\begin{equation} \label{Phi14}
\int_u^{\infty} t^{\frac{\alpha}{n}} \, \Phi^{-1}(t^{\lambda-1}) \, \frac{dt}{t} \leq C_1\, \Psi^{-1}(u^{\mu-1}) \quad \text{for all} \ u > 0 
\end{equation}
and
\begin{equation} \label{Orlicz15-2} 
u^\frac{\alpha}{n} \Phi^{-1}(\frac{r^\lambda}{u})+\int_u^{r} t^\frac{\alpha}{n} \, \Phi^{-1} (\frac{r^\lambda}{t})\, \frac{dt}{t} 
\leq C_2 \,\Psi^{-1}(\frac{r^\mu}{u}) 
\quad \text{for all} \  r> u > 0. 
\end{equation}
\begin{itemize}
\item[$({\rm i})$] If ${\Phi}^* \in {\Delta}_2$, then $I_{\alpha}$ is bounded from $M^{\Phi, \lambda}(0)$ to $M^{\Psi, \mu}(0)$, 
that is, there exists a constant $C_3 = C_3(n, C_0, C_H, C_1, C_2) \geq 1$ such that $\| I_{\alpha} f\|_{M^{\Psi, \mu}(0)} \leq C_3 \,  \|f\|_{M^{\Phi, \lambda}(0)}$ 
for all $f \in M^{\Phi, \lambda}(0)$. 
\item[$({\rm ii})$] The operator $I_{\alpha}$ is bounded from $M^{\Phi, \lambda}(0)$ to $WM^{\Psi, \mu}(0)$, that is, 
there exists a constant $c_3 = c_3 (n, c_0, C_H, C_1, C_2) \geq 1$ such that $\| I_{\alpha} f\|_{WM^{\Psi, \mu}(0)} 
\leq c_3\,  \|f\|_{M^{\Phi, \lambda}(0)}$ for all $f \in M^{\Phi, \lambda}(0)$. 
\end{itemize}
\end{theorem}
%%%%%%%%%%%%
In our earlier paper \cite[Theorem 3]{BMM2022} it was proved result under conditions \eqref{Phi14} and \eqref{Orlicz15}, 
and the latter means that
\begin{equation} \label{Orlicz15} 
\int_u^{\infty} t^\frac{\alpha}{n} \, \Phi^{-1} (\frac{r^\lambda}{t})\, \frac{dt}{t} \leq C_4 \,\Psi^{-1}(\frac{r^\mu}{u}) 
\quad \text{for all} \  u > 0,\,\,r>0. 
\end{equation}
The condition \eqref{Orlicz15} is stronger than the assumption \eqref{Orlicz15-2} because
\begin{eqnarray*}
\int_u^{\infty} t^\frac{\alpha}{n} \, \Phi^{-1} (\frac{r^\lambda}{t})\, \frac{dt}{t} &\ge& \int_u^{2u} t^\frac{\alpha}{n} \, \Phi^{-1} 
(\frac{r^\lambda}{t})\, \frac{dt}{t} \ge u^\frac{\alpha}{n} \int_u^{2u} \Phi^{-1} (\frac{r^\lambda}{2u})\, \frac{dt}{t}\\
&\ge& \frac{u^\frac{\alpha}{n}}{2} \int_u^{2u}\Phi^{-1}(\frac{r^\lambda}{u})\frac{dt}{t} = \frac{\ln{2}}{2}u^\frac{\alpha}{n}\Phi^{-1}(\frac{r^\lambda}{u})
\end{eqnarray*}
and clearly the integral in \eqref{Orlicz15-2} is smaller than the integral in \eqref{Orlicz15}. This improvement provides us with larger 
classes of Orlicz functions $\Phi$ and $\Psi,$ defining central Morrey--Orlicz spaces where the operator $I_\alpha$ is bounded. 

In the simplest case, when $\Phi (u) = u^p, \Psi (u) = u^q$ where $1 < p < q < \infty$, then the convergence of the integral 
in (\ref{Phi14}) means $p < \frac{n (1 - \lambda)}{\alpha}$ and the assumption itself gives equality 
$\frac{\alpha}{n} + \frac{\lambda - 1}{p} = \frac{\mu - 1}{q}$. Assumptions (\ref{Orlicz15-2}) and (\ref{Orlicz15}) are both 
equivalent and give the following equations: $\frac{1}{q} = \frac{1}{p} - \frac{\alpha}{n}$ and $\frac{\lambda}{p} = \frac{\mu}{q}$.
Of course, with the above assumptions, the operator $I_{\alpha}$ is bounded from $M^{p, \lambda}(0)$ to $M^{q, \mu}(0)$. 

Only later, on the Examples \ref{ex2} and \ref{ex3}, we will see that the conditions \eqref {Phi14} and \eqref {Orlicz15-2} hold but 
estimate \eqref {Orlicz15} fails, which shows that our Theorem 1 improves Theorem 3 in \cite {BMM2022}.

\vspace{2mm}
Let us comment on what we can get when the numbers $\lambda$ and $\mu$ come from ``boundaries".

%%%%%%%%Remark 1
\begin{remark}
If $\lambda = \mu = 0$ we come to the same conclusion as in \cite[Remark 4]{BMM2022}, that is, condition \eqref{Phi14} is sufficient for 
the boundedness of $I_\alpha$ from Orlicz space $L^\Phi(\mathbb{R}^n)$ to weak Orlicz space $WL^\Psi(\mathbb{R}^n).$ If, in addition 
$\Phi^{\ast} \in \Delta_2$, then $I_\alpha$ is bounded from $L^\Phi(\mathbb{R}^n)$ to $L^\Psi(\mathbb{R}^n)$. Note that in this case 
condition \eqref{Orlicz15-2} follows from \eqref{Phi14}.
\end{remark}

%%%%% Remark 2
\begin{remark}
If $\lambda = 0$ and $0 < \mu < 1$, then the condition \eqref{Orlicz15} is not satisfied, as we already mentioned in \cite[Remark 3]{BMM2022} 
and therefore the result proved in \cite{BMM2022} does not include boundedness of the Riesz potential in this case. 
On the other hand, in this case, assumption \eqref{Phi14} is stronger than \eqref{Orlicz15-2}. Indeed, 
\begin{eqnarray*}
u^\frac{\alpha}{n} \Phi^{-1}(\frac1{u}) &\le & 4 \int\limits_u^\infty t^\frac{\alpha}{n} \Phi^{-1}\bigl(\frac 1t\bigr)\frac{dt}{t} 
\le 4\,C_1 \Psi^{-1}(u^{\mu-1}) \le 4\,C_1 \Psi^{-1}(\frac{r^\mu}{u}) \quad \text{for all}\,\,\, r>u>0.
\end{eqnarray*}
Therefore, if \eqref{Phi14} holds, then $I_\alpha$ is bounded from $L^\Phi (\mathbb{R}^n)$ to $M^{\Psi,\mu}(0).$ 
In particular, when $\Phi(u) = u^p,~\Psi(u) = u^q$, $0<\frac{\alpha}{n}<\frac1p$ and $\frac{\alpha}{n}-\frac1p = \frac{\mu-1}{q}$,  
then \eqref{Phi14} holds. In fact, for all $u > 0$
$$
\int_u^\infty t^\frac{\alpha}{n} \Phi^{-1}(\frac1{t}) \frac{dt}{t} = \int\limits_u^\infty t^{\frac{\alpha}{n}-\frac 1p-1}\,dt 
= \frac1{\frac1p - \frac{\alpha}{n}} u^{\frac{\alpha}{n}-\frac 1p} = \frac{q}{1-\mu} u^\frac{\mu-1}{q} = C_1 \Psi^{-1}(u^{\mu-1}), 
$$
and we obtain \eqref{Phi14} with $C_1 = \frac{q}{1-\mu}.$ Thus, from Theorem~\ref{th_bound} we get that $I_\alpha$ 
is bounded from $L^p(\mathbb{R}^n)$ to $M^{q,\mu}(0)$. This result, in particular, was proved in \cite[Theorem 2]{BGG2007}. 
\end{remark}

%%%%%%Remark 3
\begin{remark}\label{remark3}
If $0 < \lambda < 1$ and $\mu = 0$, the conditions \eqref{Orlicz15-2} and \eqref{Orlicz15} are not satisfied. Additionally, $I_\alpha$ 
is not bounded from $M^{\Phi, \lambda}(0)$ to $L^\Psi(\mathbb{R}^n)$ by applying the necessary condition for boundedness 
of $I_\alpha$ given in \cite[Theorem 2 (ii)]{BMM2022}. 
In fact, let $R \ge 1,$ $x_R = (R, 0,..., 0) \in \mathbb{R}^n$ and $f_R(x) = \chi_{B(x_R,1)}(x).$ Following the same arguments 
as in \cite[Proposition 1]{KS17} and in \cite[Theorem 2 (ii)]{BMM2022} we obtain that 
$$
\|f_R\|_{M^{\Phi,\lambda}(0)} \le \frac1{\Phi^{-1}(\frac{v_n^\lambda}{2^n v_{n-1}}R^{\lambda n})} ~~ {\rm and} ~~ 
\|I_\alpha f_R\|_{L^\Psi(\mathbb{R}^n)} \ge \frac{2^{\alpha-n}v_n}{\Psi^{-1}(\frac1{v_n})}.
$$
Thus, 
\begin{eqnarray*}
\liminf\limits_{R\rightarrow \infty} \frac{\|I_\alpha f_R\|_{L^\Psi(\mathbb{R}^n)}}{\|f_R\|_{M^{\Phi,\lambda}(0)}} 
&\ge& 
\frac{2^{\alpha-n} v_n}{\Psi^{-1}(\frac1{v_n})} \liminf\limits_{R\rightarrow \infty} \Phi^{-1}(\frac{v_n^\lambda}{2^n v_{n-1}}R^{\lambda n}) \\
&\ge &
\frac{2^{\alpha-n} v_n}{\Psi^{-1}(\frac1{v_n})} \min(1, \frac{v_n^\lambda}{2^n v_{n-1}}) 
\liminf\limits_{R\rightarrow \infty} \Phi^{-1}(R^{\lambda n}) = \infty,
\end{eqnarray*}
and therefore $I_\alpha$ is not bounded from $M^{\Phi, \lambda}(0)$ to $L^\Psi(\mathbb{R}^n).$
\end{remark}

%%%%%%%Remark 4
\begin{remark}\label{remark4}
If $0 < \lambda = \mu < 1$, then the assumption \eqref{Orlicz15-2} does not hold. Moreover, if either 
$a = \liminf\limits_{t \rightarrow 0^+} \frac{\Phi^{-1}(t)}{\Psi^{-1}(t)} > 0$ or $b = \liminf\limits_{t \rightarrow \infty} 
\frac{\Phi^{-1}(t)}{\Psi^{-1}(t)} = \infty$, then by Theorem 2 in \cite{BMM2022} the Riesz potential $I_{\alpha}$ 
is not bounded from $M^{\Phi, \lambda}(0)$ to $M^{\Psi, \lambda}(0)$. In particular, $I_{\alpha}$ is not bounded 
from $M^{p, \lambda}(0)$ to $M^{q, \lambda}(0)$ for any $1 \leq p, q < \infty$ (see also \cite{KS17}).
There remains an unresolved case when $a = 0$ and $b < \infty$.
\end{remark}

%%%%%%%%%%%%%%% Proof 
\begin{proof} [Proof of Theorem \ref{th_bound}](i) For any $x \in B_r$ and $f \in M^{\Phi, \lambda} (0)$ we consider two disjoint subsets
\begin{equation*}
B_r^1 = \{x \in B_r \colon \Phi\left(\frac{Mf(x)}{C_0 \|f\|_{M^{\Phi, \lambda}(0)}}\right) \le |B_r|^{\lambda - 1}\},
\end{equation*}
and
\begin{equation*}
B_r^2 = \{ x \in B_r \colon \Phi\left(\frac{Mf(x)}{C_0 \|f\|_{M^{\Phi, \lambda}(0)}}\right) > |B_r|^{\lambda - 1}\}.
\end{equation*}
We estimate the Riesz potential $I_\alpha f(x)$ by a sum of two integrals
\begin{eqnarray*}
|I_\alpha f(x)|  &\le & \int_{|y| \le 2r} |f(y)||x-y|^{\alpha - n} \,dy + \int_{|y|>2r}|f(y)||x-y|^{\alpha - n} \,dy\\
 & =: &I_1f(x)+I_2f(x).
\end{eqnarray*}
For $x \in B_r^1$ and $|y| \leq 2r$ we have $|y - x| \leq |y| +|x| \leq 3r$, and so
\begin{equation*}
I_1 f(x) = \int_{|y| \le 2r} |f(y)||x-y|^{\alpha - n} \,dy \leq \int_{|y - x| \le 3 r} |f(y)||x-y|^{\alpha - n} \,dy.
\end{equation*}
By Hedberg's pointwise estimate, given in Lemma \ref{lemma2}, we obtain
\begin{equation*}
I_1f(x) \leq C_5\, |B_r|^{\frac {\alpha} n} Mf(x), ~ {\rm where} ~ C_5 = C_H \, 3^{\alpha} \,v_n^{-\alpha/n}.
\end{equation*}
This implies, for $x \in B_r^1$, that
\begin{equation*}
I_1f(x) \le C_0\, C_5 \|f\|_{M^{\Phi, \lambda}(0)} |B_r|^{\frac {\alpha} n} \Phi^{-1}(|B_r|^{\lambda-1}).
\end{equation*}
On the other hand,
\begin{align*}
\int\limits_u^\infty t^{\frac{\alpha}{n}} \Phi^{-1} (t^{\lambda-1}) \frac{dt}{t} &\ge 
\int\limits_u^{2u} t^{\frac{\alpha}{n}} \Phi^{-1} (t^{\lambda-1}) \frac{dt}{t} \ge \ln{2}~ u^{\frac{\alpha}{n}} \Phi^{-1}((2u)^{\lambda-1})\\ &\ge 
2^{\lambda-1}\ln 2~u^{\frac{\alpha}{n}} \Phi^{-1} (u^{\lambda-1}) \ge \frac{1}{4} \, u^{\frac{\alpha}{n}} \Phi^{-1} (u^{\lambda-1}),
\end{align*}
for any $u > 0$. Thus, applying assumption \eqref{Phi14} we obtain
\begin{eqnarray*}
I_1 f(x) &\le& 4 \, C_0\, C_1\,C_5 \|f\|_{M^{\Phi, \lambda}(0)} \Psi^{-1}(|B_r|^{\mu-1})\\
&\le& \frac{4}{2^{n(\mu-1)}} \,C_0\, C_1\,C_5 \|f\|_{M^{\Phi, \lambda}(0)} \Psi^{-1}(|B_{2r}|^{\mu-1})\\
&\le&
 4 \cdot 2^n \cdot C_0\, C_1\,C_5 \, \|f\|_{M^{\Phi, \lambda}(0)} \Psi^{-1}(|B_{2r}|^{\mu-1}).
\end{eqnarray*}

To estimate the second integral $I_2 f(x) $, first note that when $x \in B_r^1$ and $|y| > 2r$ we have 
$|x| < r < |y|/2$ and $|y - x| \ge |y| - |x| > |y|/2$, and so $|x - y|^{\alpha - n} < 2^{n - \alpha} |y|^{\alpha - n}$. 
Thus, following Hedberg's method, as in \cite[pp. 18--20]{BMM2022}, we obtain 
\begin{eqnarray*}
I_2 f(x) &\le& 2^{n-\alpha} \int_{|y|>2r} |f(y)||y|^{\alpha - n}\,dy = 
2^{n-\alpha} \sum\limits_{k = 0}^\infty \, \int\limits_{2^k\cdot2r<|y|\le2^{k+1}\cdot2r} |f(y)||y|^{\alpha - n}\,dy \\
&\le & 2^{n-\alpha} \sum\limits_{k = 0}^\infty (2^{k+1} r)^{\alpha-n}\int_{|y| \le 2^{k+2}r} |f(y)|\,dy.
\end{eqnarray*}
Then, from Lemma \ref{lemma1}, it follows that 
\begin{eqnarray*}
I_2 f(x) &\le &  2^{n-\alpha+1} \sum\limits_{k = 0}^\infty (2^{k+1}r)^{\alpha-n} |B_{2^{k+2}r}|^\lambda \, 
\|f\|_{\Phi,\lambda,B_{2^{k+2}r}} \|\chi_{B_{2^{k+2}r}}\|_{\Phi^\ast, \lambda,B_{2^{k+2}r}}\\
&\le & 2^{n-\alpha+1} \|f\|_{M^{\Phi,\lambda}(0)} \sum\limits_{k = 0}^\infty (2^{k+1}r)^{\alpha-n} 
|B_{2^{k+2}r}|^\lambda\, \frac{\Phi^{-1}(|B_{2^{k+2}r}|^{\lambda-1})}{|B_{2^{k+2}r}|^{\lambda-1}}\\
&=& 2^{2n-\alpha+1} v_n \,\|f\|_{M^{\Phi,\lambda}(0)} \sum\limits_{k = 0}^\infty (2^{k+1}r)^{\alpha} \, 
\Phi^{-1}(|B_{2^{k+2}r}|^{\lambda-1})\\
&=&\frac{2^{2n-\alpha+1} v_n^{1-\frac{\alpha}{n}}}{n\ln{2}} \|f\|_{M^{\Phi,\lambda}(0)} 
\sum\limits_{k = 0}^\infty |B_{2^{k+1}r}|^\frac{\alpha}{n} \Phi^{-1}(|B_{2^{k+2}r}|^{\lambda-1}) 
\int\limits_{|B_{2^{k+1}r}|}^{|B_{2^{k+2}r}|} \frac{dt}{t}\\
&\le& 2^{2n-\alpha+2} v_n^{1-\frac{\alpha}{n}} \, \|f\|_{M^{\Phi,\lambda}(0)} \sum\limits_{k = 0}^\infty 
\int\limits_{|B_{2^{k+1}r}|}^{|B_{2^{k+2}r}|} t^{\frac{\alpha}{n}} \Phi^{-1}(t^{\lambda-1})\frac{dt}{t}\\
&\le& C_6 \, \|f\|_{M^{\Phi,\lambda}(0)} \int\limits_{|B_{2r}|}^\infty t^\frac{\alpha}{n}\Phi^{-1}(t^{\lambda-1})\frac{dt}{t},
~ {\rm where} ~ C_6 = 2^{2n-\alpha+2} v_n^{1-\frac{\alpha}{n}}.
\end{eqnarray*}
Applying assumption \eqref{Phi14} we get
$$
I_2 f(x) \le C_1 \, C_6 \, \|f\|_{M^{\Phi,\lambda}(0)} \Psi^{-1} (|B_{2r}|^{\mu-1}). 
$$
Thus, for $x \in B_r^1$, we obtain
$$
|I_\alpha f(x)| \le I_1 f(x) + I_2 f(x) \le 2 \, C_7 \, \|f\|_{M^{\Phi,\lambda}(0)}\Psi^{-1} (|B_{2r}|^{\mu-1}),
$$
where $C_7 = C_1\cdot \max \{4 \cdot 2^n \cdot C_0\, C_5, C_6\}$. Since $2^{n(\mu-1)}<1$ it follows that 
\begin{eqnarray*}
\int_{B_r^1}\Psi\left(\frac{|I_\alpha f(x)|}{2 \,C_7\, \|f\|_{M^{\Phi,\lambda}(0)}}\right)\,dx \le |B_r^1| |B_{2r}|^{\mu-1}\le2^{n(\mu-1)}|B_r|^\mu<|B_r|^\mu.
\end{eqnarray*}

Let now $x \in B_r^2.$ We can write $I_\alpha f(x)$ as follows
\begin{eqnarray*}
|I_\alpha f(x)|  &\le & \int_{|x-y| \le \delta} |f(y)||x-y|^{\alpha - n} \,dy + \int_{|x-y|>\delta}|f(y)||x-y|^{\alpha - n} \,dy\\
 & =: &I_3f(x)+I_4f(x),
\end{eqnarray*}
where $\delta$ is defined in the following way
\begin{equation}\label{delta}
\Phi\left(\frac{Mf(x)}{C_0\|f\|_{M^{\Phi,\lambda}(0)}}\right) = \frac{|B_r|^\lambda}{|B_\delta|}.
\end{equation} 
Since $x \in B_r^2$ it follows that $|B_\delta|<|B_r|$.
Hedberg's pointwise estimate from Lemma 2 to $ I_3 f (x) $ gives
$$
I_3 f(x) \le C_H \, \delta^{\alpha} Mf(x) = C_H \, (\delta^n v_n)^{\alpha/n} \,v_n^{-\alpha/n} Mf(x) 
= v_n^{-\alpha/n}\,C_H \, |B_\delta|^\frac{\alpha}{n} Mf(x), 
$$
and from the assumption \eqref{Orlicz15-2} we get
$$
I_3 f(x) \le v_n^{-\alpha/n}\,C_2\,C_H \frac{\Psi^{-1}(\frac{|B_r|^\mu}{|B_\delta|})}{\Phi^{-1}(\frac{|B_r|^\lambda}{|B_\delta|})} Mf(x).
$$
Next, since equality \eqref{delta} holds it follows that 
$$
I_3 f(x) \le v_n^{-\alpha/n}\,C_0\,C_2\,C_H \, \|f\|_{M^{\Phi,\lambda}(0)}\Psi^{-1}(\frac{|B_r|^\mu}{|B_\delta|}).
$$
Applying again Hedberg's method for $I_4 f(x)$ we obtain
\begin{eqnarray*}
I_4 f(x) &=& \int_{|x-y|>\delta}|f(y)||x-y|^{\alpha-n}\,dy 
= \sum\limits_{k =0}^{\infty} \, \int\limits_{2^k\delta<|x-y|\le2^{k+1}\delta}|f(y)||x-y|^{\alpha-n}\,dy\\
&\le& \sum\limits_{k =0}^{\infty}(2^k\delta)^{\alpha-n}\int\limits_{|x-y|\le2^{k+1}\delta}|f(y)|\,dy \\
&\le& \sum\limits_{k =0}^{\infty}(2^k\delta)^{\alpha-n}\int\limits_{B_{|x|+2^{k+1}\delta}}|f(y)|\chi_{B(x,2^{k+1}\delta)}(y)\,dy,
\end{eqnarray*}
where $B_{|x|+2^{k+1}\delta}$ is the smallest ball with the centre at origin containing $B(x,2^{k+1}\delta).$
From Lemma \ref{lemma1}, using the fact that $B_{|x|+2^{k+1}\delta} \cap B(x,2^{k+1}\delta) = B(x,2^{k+1}\delta)$, we get
\begin{eqnarray*}
I_4 f(x) &\le& 2 \sum\limits_{k =0}^{\infty}(2^k\delta)^{\alpha-n} |B_{|x|+2^{k+1}\delta}|^\lambda 
\|f\|_{\Phi,\lambda,B_{|x|+2^{k+1}\delta}} \| \chi_{B(x,2^{k+1}\delta)}\|_{\Phi^\ast, \lambda,B_{|x|+2^{k+1}\delta}}\\
&\le& 2 \, \|f\|_{M^{\Phi,\lambda}(0)} \sum\limits_{k =0}^{\infty}(2^k\delta)^{\alpha-n} \,|B(x,2^{k+1}\delta)| \, 
\Phi^{-1}\left(\frac{|B_{|x|+2^{k+1}\delta}|^\lambda}{|B(x,2^{k+1}\delta)|}\right)\\
&=& \frac{2^{n+1}v_n}{n\ln{2}} \, \|f\|_{M^{\Phi,\lambda}(0)} \sum\limits_{k =0}^{\infty}(2^k\delta)^{\alpha} \,
\Phi^{-1}\left(\frac{|B_{|x|+2^{k+1}\delta}|^\lambda}{|B(x,2^{k+1}\delta)|}\right)\int\limits_{|B(x,2^k\delta)|}^{|B(x,2^{k+1}\delta)|}\frac{dt}{t}.
\end{eqnarray*}
Since $|x| \le r$ and $2^k\delta \le (\frac{t}{v_n})^\frac1n\le2^{k+1}\delta$ it follows that
\begin{eqnarray*}
|B_{|x|+2^{k+1}\delta}|&\le&v_n(r+2^{k+1}\delta)^n \le v_n\left(r+2\frac{t^\frac1n}{v_n^\frac1n}\right)^n = (v_n^\frac1nr+2t^\frac1n)^n\\
&=&(|B_r|^\frac1n+2t^\frac1n)^n \le 2^n (|B_r|^\frac1n+t^\frac1n)^n \le 4^n \max\{|B_r|,t\}.
\end{eqnarray*}
So using the concavity of $\Phi^{-1}$, we get
\begin{eqnarray*}
I_4 f(x) &\le& \frac{2^{n+1}v_n^{1-\frac{\alpha}{n}}}{n\ln{2}} \, \|f\|_{M^{\Phi,\lambda}(0)} 
\sum\limits_{k =0}^{\infty}\int\limits_{|B(x,2^k\delta)|}^{|B(x,2^{k+1}\delta)|}t^\frac{\alpha}{n} 
\Phi^{-1}\left(\frac{4^{\lambda n}(\max\{|B_r|,t\})^\lambda}{t}\right)\frac{dt}{t}\\
&\le& \frac{4^{\lambda n}2^{n+1}v_n^{1-\frac{\alpha}{n}}}{n\ln{2}} \, \|f\|_{M^{\Phi,\lambda}(0)} 
\int\limits_{|B_\delta|}^\infty t^\frac{\alpha}{n}\Phi^{-1}\left(\frac{(\max\{|B_r|,t\})^\lambda}{t}\right)\frac{dt}{t}\\
&\le& C_8 \, \|f\|_{M^{\Phi,\lambda}(0)} \left[ \int\limits_{|B_\delta|}^{|B_r|} t^\frac{\alpha}{n}\Phi^{-1}\left(\frac{|B_r|^\lambda}{t}\right)\frac{dt}{t} + 
\int\limits_{|B_r|}^{\infty} t^\frac{\alpha}{n}\Phi^{-1}(t^{\lambda-1})\frac{dt}{t} \right], 
\end{eqnarray*}
where $C_8 = \frac{4^{\lambda n}2^{n+1} v_n^{1-\frac{\alpha}{n}}}{n\ln{2}} \leq \frac{4^n \cdot 2^{n+2} \cdot v_n^{1-\frac{\alpha}{n}}}{n}$. 
Based on the assumptions of \eqref {Phi14}, \eqref {Orlicz15-2} and the fact that $|B_\delta| < |B_r|$ we get
\begin{eqnarray*}
I_4 f(x) &\le& C_8 \,\|f\|_{M^{\Phi,\lambda}(0)} \left[C_2\, \Psi^{-1}\left(\frac{|B_r|^\mu}{|B_\delta|}\right)+C_1\, \Psi^{-1}(|B_r|^{\mu-1})\right]\\
&\le& C_{8} \, (C_2 + C_1) \, \|f\|_{M^{\Phi,\lambda}(0)} \Psi^{-1}\left(\frac{|B_r|^\mu}{|B_\delta|}\right).
\end{eqnarray*}
Thus, for $x\in B_r^2$ we obtain
\begin{eqnarray*}
| I_\alpha f(x)| \le I_3 f(x)+I_4 f(x) \le C_9 \, \|f\|_{M^{\Phi,\lambda}(0)} \Psi^{-1}\left(\frac{|B_r|^\mu}{|B_\delta|}\right),
\end{eqnarray*}
with $C_9 = v_n^{-\alpha/n}\,C_0\,C_2\,C_H + C_8 \, (C_1 + C_2)$. Then
\begin{eqnarray*}
\int_{B_r^2}\Psi\Bigl(\frac{|I_\alpha f(x)|}{C_9 \, \|f\|_{M^{\Phi,\lambda}(0)}}\Bigr)\,dx &\le& \int_{B_r}\frac{|B_r|^\mu}{|B_\delta|} \,dx\\
&=& |B_r|^{\mu-\lambda}\int_{B_r}\Phi\Bigl(\frac{Mf(x)}{C_0 \|f\|_{M^{\Phi,\lambda}(0)} }\Bigr)\,dx\\
&\le& |B_r|^{\mu-\lambda}\int_{B_r}\Phi\Bigl(\frac{Mf(x)}{\|Mf\|_{M^{\Phi,\lambda}(0)} }\Bigr)\,dx \le |B_r|^\mu.
\end{eqnarray*}
Finally, since $B_r = B_r^1 \cup B_r^2$ and the last two sets are disjoint, and by the convexity of $\Psi$ it follows that
\begin{eqnarray*}
\int\limits_{B_r}\Psi\Bigl(\frac{|I_\alpha f(x)|}{C_3 \, \|f\|_{M^{\Phi,\lambda}(0)}}\Bigr)\,dx 
&=&
\int\limits_{B_r^1}\Psi\Bigl(\frac{|I_\alpha f(x)|}{C_3 \, \|f\|_{M^{\Phi,\lambda}(0)}}\Bigr)\,dx + 
\int\limits_{B_r^2}\Psi\Bigl(\frac{|I_\alpha f(x)|}{C_3 \, \|f\|_{M^{\Phi,\lambda}(0)}}\Bigr)\,dx \\
&\le&\frac12\int\limits_{B_r^1}\Psi\Bigl(\frac{|I_\alpha f(x)|}{2 \,C_7 \, \|f\|_{M^{\Phi,\lambda}(0)}}\Bigr)\,dx
+\frac12\int\limits_{B_r^2}\Psi\Bigl(\frac{|I_\alpha f(x)|}{C_9 \, \|f\|_{M^{\Phi,\lambda}(0)}}\Bigr)\,dx\\
&\le& |B_r|^\mu,
\end{eqnarray*}
where $C_3 = 2\max\{2 \, C_7, C_9\}.$ Hence, $\|I_\alpha f\|_{M^{\Psi,\mu}(0)} \le C_3 \,\|f\|_{M^{\Phi,\lambda}(0)}$.

\vspace{3mm}

(ii) Similarly to the previous case, we will present $B_r$ as a union of two disjoint subsets $B_r = B_r^1 \cup B_r^2$, 
where $B_r^1$ and $B_r^2$ are defined in the same way as in the first part of the proof with respect to the constant 
$c_0$, that is,
\begin{equation*}
B_r^1 = \{x \in B_r \colon \Phi\left(\frac{Mf(x)}{c_0 \|f\|_{M^{\Phi, \lambda}(0)}}\right) \le |B_r|^{\lambda - 1}\},
\end{equation*}
and
\begin{equation*}
B_r^2 = \{ x \in B_r \colon \Phi\left(\frac{Mf(x)}{c_0 \|f\|_{M^{\Phi, \lambda}(0)}}\right) > |B_r|^{\lambda - 1}\}.
\end{equation*}
Then we get
$$
\Psi\Bigl(\frac{|I_\alpha f(x)|}{c_3 \,\|f\|_{M^{\Phi,\lambda}(0)}}\Bigr) 
\le \frac12 \Psi\Bigl(\frac{|I_\alpha (f \chi_{B_r^1})(x)|}{4 \,c_7 \, \|f\|_{M^{\Phi,\lambda}(0)}}\Bigr)
+\frac12 \Psi\Bigl(\frac{|I_\alpha (f \chi_{B_r^2})(x)|}{2\,c_9 \, \|f\|_{M^{\Phi,\lambda}(0)}}\Bigr) =: \frac12(I_5+I_6),
$$
where $c_3 = 2 \,\max\{4\,c_7, 2\,c_9\}, c_7 = C_1 \max\{4 \cdot 2^n \cdot c_0\,C_5, C_6\}, 
c_9 = v_n^{-\alpha/n}\,c_0\,C_2\,C_H +C_8(C_1+C_2)$. 
We follow the same calculations as in the proof of Theorem 3(ii) in \cite{BMM2022} and we get
$$
d\left(\Psi\left(\frac{|I_\alpha f(x)|}{c_3 \,\|f\|_{M^{\Phi,\lambda}(0)}}\right), u\right) \le d(I_5,u)+d(I_6,u)
$$
and 
$$
\sup\limits_{u>0}\frac{\Psi(u)}{|B_r|^\mu} \,d\left(\frac{|I_\alpha f(x)|}{c_3 \, \|f\|_{M^{\Phi,\lambda}(0)}},u\right) 
\le \sup\limits_{u>0}\frac{u}{|B_r|^\mu} \,d(I_5,u) + \sup\limits_{u>0}\frac{u}{|B_r|^\mu} \,d(I_6,u),
$$
where we used the property $\Psi(u) \,d(g,u) = v\, d(g, \Psi^{-1}(v)) = v\, d(\Psi(g),v)$ for any $u>0$ with $v = \Psi(u).$

From the first part of the proof of this theorem for any $r>0$ we have
$$
I_5 = \Psi\Bigl(\frac{|I_\alpha (f \chi_{B_r^1})(x)|}{2 \cdot 2 c_7 \, \|f\|_{M^{\Phi,\lambda}(0)}}\Bigr) 
\le \frac 12 |B_{2r}|^{\mu-1}<\frac12|B_r|^{\mu-1}
$$
and 
$$
\sup\limits_{u>0}\frac{u}{|B_r|^\mu} \,d(I_5,u) \le \frac12\sup\limits_{u>0}\frac{u}{|B_r|^\mu} \,d(|B_r|^{\mu-1},u) 
= \frac12\sup\limits_{u>0} u\, d(\frac1{|B_r|},u) \le \frac12.
$$
For $I_6$ from the first part of the proof of this theorem we obtain
$$
I_6 = \Psi\Bigl(\frac{|I_\alpha (f \chi_{B_r^2})(x)|}{2\,c_9 \|f\|_{M^{\Phi,\lambda}(0)}}\Bigr) 
\le \frac12 \frac{|B_r|^\mu}{|B_\delta|},
$$
where $\delta$ is defined as in \eqref{delta} with respect to $c_0$, that is, 
\begin{equation*}
\Phi\left(\frac{Mf(x)}{c_0\|f\|_{M^{\Phi,\lambda}(0)}}\right) = \frac{|B_r|^\lambda}{|B_\delta|}.
\end{equation*}
Thus,
$$
I_6 \le \frac12 |B_r|^{\mu-\lambda}\Phi\left(\frac{Mf(x)}{c_0 \|f\|_{M^{\Phi,\lambda}(0)}}\right)
$$
and doing the same calculations as in the proof of Theorem 3(ii) in \cite{BMM2022} we get
$$
\sup\limits_{u>0}\frac{u}{|B_r|^\mu}d(I_6,u) \le \frac 12.
$$
Hence,
$$
\sup\limits_{u>0}\frac{\Psi(u)}{|B_r|^\mu}d\left(\frac{|I_\alpha f(x)|}{c_3 \|f\|_{M^{\Phi,\lambda}(0)}}, u\right) \le 1
$$
and $\|I_\alpha f\|_{WM^{\Psi,\mu}(0)} \le c_3 \|f\|_{M^{\Phi,\lambda}(0)}.$
\end{proof}

Below we present examples for our Theorem \ref{th_bound}. In our earlier paper \cite{BMM2022} we have shown that 
Example \ref{ex1} holds under conditions \eqref{Phi14} and \eqref{Orlicz15}, which clearly means that it also holds 
under conditions \eqref{Phi14} and \eqref{Orlicz15-2} of Theorem \ref{th_bound}.
%%%%%%%%%%%%%%%%%%% Example 1
\begin{example} \label{ex1}
Let $0 < \alpha < n, 0 \le \lambda < 1, 1 < p < \frac{n(1-\lambda)}{\alpha}, 0 \le a \le \sqrt{1-\frac1p}-(1-\frac1p)$ and 
\begin{equation*} 
\Phi^{-1}(u) = 
\begin{cases}
u^{\frac 1 p} & \text{for}~0 \leq u \leq 1,\\
u^{\frac 1 p} \left(1 + \ln u \right)^{-a} & \text{for}~u \geq1,
\end{cases}
\quad \Psi^{-1}(u) =  u^{\frac{1}{q}} ~~ \text{with} ~ 1 < p < q < \infty.
\end{equation*}
If $\frac{1}{q} = \frac{1}{p} - \frac{\alpha}{n}, \frac{\lambda}{p} = \frac{\mu}{q}$, then conditions \eqref{Phi14} and \eqref{Orlicz15-2} 
of Theorem \ref{th_bound} are satisfied, and the Riesz potential $I_\alpha$ is bounded from $M^{\Phi, \lambda}(0)$ to $M^{\Psi, \mu}(0)$. 
In particular, if $a=0$ we get the Spanne--Peetre type result \cite{Pe69} proved in \cite[Proposition 1.1]{FLL08}, that is, 
the Riesz potential $I_\alpha$ is bounded from $M^{p,\lambda}(0)$ to $M^{q,\mu}(0)$ under the conditions 
$1 < p < \frac{n(1-\lambda)}{\alpha}, 0 \leq \lambda < 1, \frac{1}{q} = \frac 1 p - \frac{\alpha}{n}$ and $\frac{\lambda}{p} 
= \frac{\mu}{q}$.
\end{example}

The next two examples satisfy conditions \eqref{Phi14} and \eqref{Orlicz15-2}, but the requirment \eqref{Orlicz15} does not hold 
for them. %In section \ref{appendix} we present detailed calsulations for these examples. 

%%%%%%%%%%%%%%%%%%% Example 2
\begin{example} \label{ex2}
Let $0<\alpha<n,$~$0 < \lambda<1,$~$1<p_1<p_2< \frac{n(1-\lambda)}{\alpha},~1<q_1<q_2<\infty$ and
$$
\Phi(u) = \max (u^{p_1}, u^{p_2}), \qquad \Psi(u) = \max(u^{q_1}, u^{q_2}).
$$ 
If $\frac 1{p_1}-\frac{\alpha}{n} = \frac 1{q_1},~\frac{\lambda}{p_1} < \frac{\mu}{q_1}$ and $\frac 1{p_2}-\frac{\alpha}{n} 
= \frac 1{q_2},~\frac{\lambda}{p_2} = \frac{\mu}{q_2}$ then conditions \eqref{Phi14} and \eqref{Orlicz15-2} of 
Theorem \ref{th_bound} are satisfied and the Riesz potential $I_\alpha$ is bounded from $M^{\Phi, \lambda}(0)$ to 
$M^{\Psi, \mu}(0).$
\end{example}

%%%%%%%%%%%%%%%%%%% Example 3
\begin{example} \label{ex3}

Let $0<\alpha<n,$~$0 < \lambda, \mu<1,$~$1<p_1<p_2< \infty,~1<q_1<q_2<\infty,~a, b>0$ and
\begin{equation*}
\Phi^{-1}(u) = 
\begin{cases}
u^{\frac 1{p_1}}(1-\ln{u})^a \qquad  &\text{for}~ 0<u \le 1,\\
u^{\frac 1{p_2}}(1+\ln{u})^{-b} \qquad &\text{for}~ u \ge 1,
\end{cases}
\end{equation*}
\begin{equation*}
\Psi^{-1}(u) =
\begin{cases}
u^{\frac 1{q_1}}(1-\frac{1-\lambda}{1-\mu}\ln{u})^a \qquad  &\text{for}~ 0<u \le 1,\\
u^{\frac 1{q_2}} \qquad &\text{for}~ u \ge 1.
\end{cases}
\end{equation*}
If $\frac 1{p_1}-\frac{\alpha}{n} = \frac 1{q_1},$ $\frac 1{p_2}-\frac{\alpha}{n} = \frac 1{q_2},~\frac{\lambda}{p_2} 
= \frac{\mu}{q_2},~\frac{\lambda}{p_1} < \frac{\mu}{q_1}$ and
$0<a \le \frac{1-\mu}{1-\lambda}(\frac1{q_1}-\frac1{q_2}),\,\, 0<b \le \frac1{p_2}$, 
then conditions \eqref{Phi14} and \eqref{Orlicz15-2} of Theorem \ref{th_bound} are satisfied and the Riesz potential 
$I_\alpha$ is bounded from $M^{\Phi, \lambda}(0)$ to $M^{\Psi, \mu}(0).$
\end{example}
%%%%%%%%%%%%%%%%%%%%%%%%

%%%%%%%%%%%%%%%%%%%%%%%%%%%%%%%%%%%%%%%%%% Section 3
\section{Two properties of central Morrey--Orlicz spaces} \label{sec:Imbed}
%%%%%%%%%%%%%%%%%%%%%%%%%%%%%%%%%%%%%%%%%%%%%%%
Properties of Morrey and central Morrey spaces were considered by several authors (for example V. I. Burenkov, 
V. S. Guliyev, E. Nakai, Y. Sawano and others). Here we will present some properties of central Morrey--Orlicz 
spaces. It is known that $M^{p, \lambda}(0) \neq \{0\}$ if and only if $\lambda \geq 0$ (see \cite{BG04}). In the next 
proposition we describe when the central Morrey--Orlicz space $M^{\Phi,\lambda}(0)$ is nontrivial.

\begin{proposition}\label{nontrivial}
Let $\Phi$ be an Orlicz function and $\lambda \in \mathbb{R}.$ The space $M^{\Phi,\lambda}(0) \neq \{0\}$ if 
and only if $\lambda \ge 0.$
\end{proposition}
\begin{proof}
Let first $\lambda<0$ and $f \in M^{\Phi,\lambda}(0),$ such that $f \not\equiv 0.$ Then 
$$
\sup\limits_{r>0} \Bigl\|\frac{f}{\|f\|_{M^{\Phi,\lambda}(0)}}\Bigr\|_{\Phi,\lambda,B_r} = 1
$$
and therefore
$$
\Bigl\|\frac{f}{\|f\|_{M^{\Phi,\lambda}(0)}}\Bigr\|_{\Phi,\lambda,B_r} \le 1 \,\,\text{for all}\,\,r>0.
$$
Thus,
$$
\frac1{|B_r|^\lambda}\int\limits_{B_r}\Phi\left(\frac{|f(x)|}{\|f\|_{M^{\Phi,\lambda}(0)}}\right)\,dx \le 1\,\,\text{for all}\,\, r>0.
$$
On the other hand, there exists $t_0>0,$ such that $\int\limits_{B_{t_0}} \Phi\left(\frac{|f(x)|}{\|f\|_{M^{\Phi,\lambda}(0)}}\right)\,dx>0$ 
and for any $r>t_0$ and $\lambda<0$ we have
$$
\int\limits_{B_{t_0}} \Phi\left(\frac{|f(x)|}{\|f\|_{M^{\Phi,\lambda}(0)}}\right)\,dx 
\le \int\limits_{B_r} \Phi\left(\frac{|f(x)|}{\|f\|_{M^{\Phi,\lambda}(0)}}\right)\,dx \le |B_r|^\lambda \rightarrow 0 \,\, \text{as}\,\, r \rightarrow \infty,
$$
which means that $f(x) = 0$ on $B_{t_0}$ and we are done.

Let now $\lambda \ge 0.$ Then we will show that there exists $f \in M^{\Phi,\lambda}(0),$ such that $f \not \equiv 0$. 
We follow ideas from \cite[Proposition 1]{KS17} and consider function $f_R(x) = \chi_{B(x_R,1)}(x),$ where $R > 1$ 
and $x_R = (R,0,...,0).$ We will show that $f \in M^{\Phi,\lambda}(0)$ for any $\lambda \ge 0.$ In our previous paper 
we have shown that 
$$
\|f_R\|_{M^{\Phi,\lambda}(0)} = \sup\limits_{r>R-1}\frac1{\Phi^{-1}(\frac{|B_r|^\lambda}{|B_r \cap B(x_R,1)|})},
$$
for details we refer to the proof of Theorem 2 in \cite{BMM2022}. Since $|B_r \cap B(x_R,1)| \le |B(x_R,1)| = v_n$ it follows 
that $\frac{|B_r|^\lambda}{|B_r \cap B(x_R,1)|} \ge \frac{|B_r|^\lambda}{v_n}$ and $\frac1{\Phi^{-1}\bigl(\frac{|B_r|^\lambda}{|B_r \cap B(x_R,1)|})} 
\le \frac 1{\Phi^{-1}\bigl(\frac{|B_r|^\lambda}{v_n }\bigr)}.$ Thus,
\begin{eqnarray*}
\|f_R\|_{M^{\Phi,\lambda}(0)} = \sup\limits_{r>R-1} \frac1{\Phi^{-1}\bigl(\frac{|B_r|^\lambda}{|B_r \cap B(x_R,1)|}\bigr)} 
&\le& \sup\limits_{r>R-1}\frac1{\Phi^{-1}\bigl(\frac{|B_r|^\lambda}{v_n }\bigr)}=\frac1{\Phi^{-1}\bigl(\frac{|B_{R-1}|^\lambda}{v_n }\bigr)},
\end{eqnarray*}
where the last equality is true since $\lambda \ge 0.$ Therefore, $f_R \in M^{\Phi,\lambda}(0).$
\end{proof}

Next we consider inclusion properties of central Morrey--Orlicz spaces. In the case of classical Morrey and classical central 
Morrey spaces it is known that if $1 \leq p < q < \infty, 0 \leq \mu < \lambda < 1$ and 
 $\frac{1 - \lambda}{p} = \frac{1 - \mu}{q}$, then 
 \begin{equation} \label{embeddings}
 M^{q, \mu}({\mathbb R^n}) \overset{1}{\hookrightarrow} M^{p, \lambda}({\mathbb R^n})  
 \quad {\rm and} \quad M^{q, \mu}(0) \overset{1}{\hookrightarrow} M^{p, \lambda}(0).
 \end{equation}
 Both inclusions are proper (see, for example, \cite{GHI18}). We also note that the second embedding in (\ref{embeddings}) 
 is also true for $1 < \lambda < \mu$. We have shown in \cite{BMM2022} that the embeddings (\ref{embeddings}) follow 
 by the Hölder--Rogers inequality with $\frac{q}{p} > 1$. In the next theorem we present inclusion properties of central 
 Morrey--Orlicz spaces. 
 %%%%%%%% Proposition 2 - Inclusion
\begin{proposition}
Let $\Phi$ and $\Psi$ be Orlicz functions, $0 \le \lambda, \mu<1$. Then $M^{\Psi, \mu}(0) \hookrightarrow M^{\Phi, \lambda}(0)$ 
if and only if there are constants $A_1, A_2>0,$ such that 
\begin{itemize}
\item[(i)] $\Phi(\frac{u}{A_1}) \le \Psi(u)^{\frac{\lambda-1}{\mu-1}}~ \text{for all}~u>0$
and
\item[(ii)] $\Phi(\frac{u}{A_2}) \le \Psi(u) r^{\lambda-\mu}~ \text{for all}~ u,r>0,$ satisfying $\Psi^{-1}(r^{\mu-1})<u$.
\end{itemize}
\end{proposition}

\begin{proof}
Let first $f \in M^{\Psi, \mu}(0)$, $f \not\equiv 0,$ $B_r$ be any open ball in $\mathbb{R}^n$ and functions $\Phi$ and $\Psi$ 
satisfy conditions (i) and (ii). Then 
$$\frac 1{|B_r|^\mu}\int_{B_r} \Psi\left(\frac{|f(x)|}{\|f\|_{M^{\Psi, \mu}(0)}}\right)\,dx \le 1$$ 
and so $\Psi \left(\frac{|f(x)|}{\|f\|_{M^{\Psi, \mu}(0)}} \right)<\infty$~ a.e. in $B_r$. We divide $B_r$ into two disjoint subsets
\begin{equation*}
B_r^3:=\Bigl\{x \in B_r: |B_r|^{\mu-1} < \Psi\left(\frac{|f(x)|}{\|f\|_{M^{\Psi, \mu}(0)}}\right)<\infty~ \text{a.e.}\Bigr\}
\end{equation*}
and
\begin{equation*}
B_r^4:=\Bigl\{x \in B_r: |B_r|^{\mu-1} \ge \Psi\left(\frac{|f(x)|}{\|f\|_{M^{\Psi, \mu}(0)}}\right) ~ \text{a.e.} \Bigr\}.
\end{equation*}
Let us denote $t = |B_r|$ and $u = \frac{|f(x)|}{\|f\|_{M^{\Psi, \mu}(0)}}$. Then, for $x \in B_r^3$ we have $0<t^{\mu-1}<\Psi(u)$ 
and from (ii) it follows that
\begin{eqnarray*}
& & \frac 1{|B_r|^\lambda} \int_{B_r^3} \Phi\left(\frac{|f(x)|}{A_2\|f\|_{M^{\Psi, \mu}(0)}}\right) \,dx 
\le \frac{1}{|B_r|^\lambda}\int_{B_r^3}\Psi\left(\frac{|f(x)|}{\|f\|_{M^{\Psi, \mu}(0)}}\right)|B_r|^{\lambda-\mu} \,dx \\
&=& 
\frac{1}{|B_r|^\mu} \int_{B_r^3}\Psi\left(\frac{|f(x)|}{\|f\|_{M^{\Psi, \mu}(0)}}\right)\,dx
 \le \frac{1}{|B_r|^\mu} \int_{B_r}\Psi\left(\frac{|f(x)|}{\|f\|_{M^{\Psi, \mu}(0)}}\right)\,dx \le 1.
\end{eqnarray*}
 For  $x \in B_r^4$ we get $0<\Psi(u) \le t^{\mu-1}$ and from (i) it follows that
\begin{eqnarray*}
\frac 1{|B_r|^\lambda} \int_{B_r^4} \Phi\left(\frac{|f(x)|}{A_1\|f\|_{M^{\Psi, \mu}(0)}}\right) \,dx 
&\le & \frac{1}{|B_r|^\lambda}\int_{B_r^4}\Psi\left(\frac{|f(x)|}{\|f\|_{M^{\Psi, \mu}(0)}}\right)^{\frac{\lambda-1}{\mu-1}}\,dx \\
&\le & 
\frac{1}{|B_r|^\lambda} \int_{B_r^4}\left(|B_r|^{\mu-1}\right)^\frac{\lambda-1}{\mu-1}\,dx = \frac{|B_r^4|}{|B_r|} \le 1.
\end{eqnarray*}

Thus,
$$
\frac 1 {|B_r|^\lambda} \int_{B_r} \Phi\left(\frac{|f(x)|}{2 \max(A_1,A_2)\|f\|_{M^{\Psi, \mu}(0)}}\right)\,dx
$$
$$
\le \frac{1}{2|B_r|^\lambda} \left[\int_{B_r^3} \Phi \left(\frac{|f(x)|}{A_1\|f\|_{M^{\Psi, \mu}(0)}}\right) \,dx 
+ \int_{B_r^4}\Phi\left(\frac{|f(x)|}{A_2 \|f\|_{M^{\Psi, \mu}(0)}}\right)\,dx \right] \le 1
$$
and so $\|f\|_{M^{\Phi,\lambda}(0)} \le 2 \max(A_1,A_2)\|f\|_{M^{\Psi, \mu}(0)}$.

Let now $\|f\|_{M^{\Phi,\lambda}(0)} \le C \|f\|_{M^{\Psi,\mu}(0)}$ for any $f\in M^{\Psi,\mu}(0)$ and some 
constant $C>0.$ First for any $t>0$ we consider $f_t(x) = \chi_{B_t}(x).$ Then, using Lemma \ref{lemma1} (iii) we obtain
$$
\frac{1}{\Phi^{-1}(|B_t|^{\lambda-1})} \le \frac{C}{\Psi^{-1}(|B_t|^{\mu-1})},
$$
which also means that $\Psi^{-1}(s^{\mu-1}) \le C \Phi^{-1}(s^{\lambda-1})$ for all $s>0$ or $s^{\mu-1} 
\le \Psi(C\Phi^{-1}(s^{\lambda-1}))$.
By change of variables $\Phi^{-1}(s^{\lambda-1}) = u$ we get 
$$
\Phi(u)^{\frac{\mu-1}{\lambda-1}} \le \Psi(Cu) ~{\rm for \,all} ~ u>0,
$$
so we have condition (i) with $A_1 = C$.

For the proof of the second part, i.e. to prove the necessity of the condition (ii), we refer to the proof of Lemma 4.12 and Theorem 4.1 
in \cite{Kita2014} and to Theorem 4.9 and Lemma 4.10 in \cite{Na08a}.
\end{proof}

%%%%%%%%%%%%%%%%%%%%%%%%%%%%%%%%%%%%%%%%%%%%%%%

\vspace{3mm}

\noindent
{\footnotesize Department of Engineering Sciences and Mathematics\\
Lule\r{a} University of Technology, SE-971 87 Lule\r{a}, Sweden\\
~{\it E-mail address: {\tt evgeniya.burtseva@ltu.se}}\\
}

\vspace{-2mm}

\noindent
{\footnotesize Institute of Mathematics, Pozna\'n University of Technology\\
ul. Piotrowo 3a, 60-965 Pozna\'n, Poland\\
~{\it E-mail address: {\tt lech.maligranda@put.poznan.pl}}\\
and\\
{\footnotesize Department of Engineering Sciences and Mathematics\\
Lule\r{a} University of Technology, SE-971 87 Lule\r{a}, Sweden\\
~{\it E-mail address: {\tt lech.maligranda@associated.ltu.se}}\\
}}

\end{document}